\documentclass[12pt,reqno]{amsart}

\usepackage[left=3.3cm,right=3.3cm]{geometry}
\usepackage{amssymb}
\usepackage{graphicx}
\usepackage{amscd}
\usepackage[pagebackref]{hyperref}
\usepackage{color}
\usepackage{tabularx}
\usepackage[table]{xcolor}
\usepackage{float}
\usepackage{graphics,amsmath,amssymb}
\usepackage{amsthm}
\usepackage{amsfonts}
\usepackage{latexsym}
\usepackage{epsf}
\usepackage{xifthen}
\usepackage{mathrsfs}
\usepackage{dsfont}
\usepackage{makecell}
\usepackage{subfig}
\usepackage{amsmath}
\allowdisplaybreaks[4]
\usepackage{listings}
\usepackage{etoolbox}
\usepackage{fancyhdr}
\usepackage{pdflscape}
\usepackage[title,toc,titletoc]{appendix}
\usepackage{enumitem}
\usepackage[noadjust]{cite}
\usepackage{tikz}
\usetikzlibrary{automata,positioning,arrows}
\usepackage{young}
\usepackage[object=vectorian]{pgfornament} %%  http://altermundus.com/pages/tkz/ornament/index.html
\usepackage{lipsum,tikz}
\usepackage{multirow}
\usepackage[OT2,T1]{fontenc}
\usepackage{mathtools}
\usepackage{ytableau}

\hypersetup{
	colorlinks=true, %set true if you want colored links
	linktoc=all, %set to all if you want both sections and subsections linked
	linkcolor=blue} %choose some color if you want links to stand out

\numberwithin{equation}{section}

\theoremstyle{theorem}
\newtheorem{theorem}{Theorem}[section]
\newtheorem*{theorem*}{Theorem}

\newtheorem{lemma}[theorem]{Lemma}

\newtheorem{innercustomgeneric}{\customgenericname}
\providecommand{\customgenericname}{}
\newcommand{\newcustomtheorem}[2]{%
	\newenvironment{#1}[1]
	{%
		\renewcommand\customgenericname{#2}%
		\renewcommand\theinnercustomgeneric{##1}%
		\innercustomgeneric
	}
	{\endinnercustomgeneric}
}
\newcustomtheorem{ctheorem}{Theorem}
\newcustomtheorem{clemma}{Lemma}

\theoremstyle{definition}

\newtheorem*{example*}{Example}
\newtheorem*{examples*}{Examples}

\newtheorem*{remark*}{Remark}
\newtheorem*{remarks*}{Remarks}
\newtheorem*{note*}{Note}

\newtheoremstyle{named}{}{}{\itshape}{}{\bfseries}{.}{.5em}{\thmnote{#3} #1}
\theoremstyle{named}

\newtheoremstyle{customized}{}{}{\itshape}{}{\bfseries}{.}{.5em}{\thmnote{#3}}
\theoremstyle{customized}
\newtheorem*{customizedempty}{}

\DeclareMathAlphabet{\mydutchcal}{U}{dutchcal}{m}{n}

\newcommand{\qbinom}[2]{{\genfrac{[}{]}{0pt}{}{#1}{#2}}}

\newcommand{\qangle}[2]{{\genfrac{\langle}{\rangle}{0pt}{}{#1}{#2}}}

\newcommand{\LHS}{\operatorname{LHS}}
\newcommand{\RHS}{\operatorname{RHS}}

\newcommand{\SYT}{\operatorname{SYT}}
\newcommand{\Des}{\operatorname{Des}}
\newcommand{\des}{\operatorname{des}}
\newcommand{\maj}{\operatorname{maj}}
\newcommand{\comaj}{\operatorname{comaj}}

\title[Cigler's conjecture on $q$-Hoggatt numbers]{Proof of Cigler's conjecture on $q$-Hoggatt numbers}

\author[S. Chern]{Shane Chern}
\address[S. Chern]{Fakult\"at f\"ur Mathematik, Universit\"at Wien, Oskar-Morgenstern-Platz 1, Wien 1090, Austria}
\email{chenxiaohang92@gmail.com, xiaohangc92@univie.ac.at}

\author[W. Shi]{Wenle Shi}
\address[W. Shi]{School of Mathematical Sciences, Dalian University of Technology, Dalian 116024, P.R. China}
\email{shi-wenle@hotmail.com}

\date{}



\begin{document}
	
\sloppy

\begin{abstract}
	We prove the nonnegativity and palindromicity of a family of polynomials arising from $q$-Hoggatt numbers. The nonnegativity is derived from Stanley's $(P,\omega)$-partition theory through a standard Young tableau formula, while the palindromicity is proved by an involution on rectangular standard Young tableaux. Our result confirms a conjecture of Cigler.
\end{abstract}

\maketitle

\section{Introduction}

A polynomial $p(x)$ is \emph{palindromic} if the sequence of its coefficients reads the same forwards from the lowest power to the highest as it does backwards. In other words, $p(x)$ satisfies the condition
\begin{align*}
	p(x) = x^N p(x^{-1}),
\end{align*}
where $N$ is the sum of the lowest and highest exponents. A classic example of palindromic polynomials is the \emph{$q$-binomial coefficient}
\begin{align*}
	\qbinom{n}{k}_q:=\begin{cases}
		\dfrac{(q;q)_n}{(q;q)_k(q;q)_{n-k}}, & \text{if $0\le k\le n$},\\[10pt]
		0, & \text{otherwise},
	\end{cases}
\end{align*}
where the \emph{$q$-Pochhammer symbol} is defined by
\begin{align*}
	(a;q)_m&:=\prod_{j=0}^{m-1} (1-a q^j).
\end{align*}
As a polynomial in $q$, the $q$-binomial coefficient $\qbinom{n}{k}_q$ has lowest exponent $0$ and highest exponent $k(n-k)$, while it is a standard result \cite[p.~353, eq.~(I.47)]{GR2004} that
\begin{align*}
	\qbinom{n}{k}_{q^{-1}} = q^{k(k-n)} \qbinom{n}{k}_q.
\end{align*}

In \cite{Cig2021}, Cigler introduced a $q$-analog of Hoggatt numbers originating from \cite{FA1989}, and conjectured a palindromic property for a relevant series. Let
\begin{align*}
	\langle n\rangle_q^{(d)} := \qbinom{n+d-1}{d}_q,\qquad\qquad \langle n\rangle_q^{(d)}! := \prod_{j=1}^n \langle j\rangle_q^{(d)}.
\end{align*}
Cigler's \emph{$q$-Hoggatt numbers} are defined by
\begin{align}\label{eq:q-Hoggatt-def}
	\qangle{n}{k}_q^{(d)} := \frac{\langle n\rangle_q^{(d)}!}{\langle k\rangle_q^{(d)}! \langle n-k\rangle_q^{(d)}!} = \prod_{j=0}^{d-1} \frac{\qbinom{n+j}{k}_q}{\qbinom{k+j}{k}_q}.
\end{align}
For $d=1$, $2$, and $3$, the $q$-Hoggatt numbers $\qangle{n}{k}_q^{(d)}$ reduce to the $q$-binomial coefficients, the $q$-Narayana numbers, and the $q$-Baxter numbers, respectively.

Cigler \cite[Conjecture~10]{Cig2021} proposed the following conjecture  on $q$-Hoggatt numbers.

\begin{customizedempty}[Cigler's Conjecture]
	Let $d$ and $k$ be any positive integers. The series
	\begin{align}
		C_k^{(d)}(x,q) := (x;q)_{kd+1} \sum_{n\ge 0} \qangle{n+k}{k}_q^{(d)} x^n
	\end{align}
	is a bivariate polynomial in $x$ and $q$ with nonnegative coefficients, with the highest exponent of $x$ being $(k-1)(d-1)$. Writing
	\begin{align*}
		C_k^{(d)}(x,q) =: \sum_{j=0}^{(k-1)(d-1)} c_{k,j}^{(d)}(q) x^j,
	\end{align*}
	it is further true that $c_{k,j}^{(d)}(q)$ is palindromic in $q$ for every $k$.
\end{customizedempty}

The main purpose of this note is to prove Cigler's conjecture by means of Stanley's $(P,\omega)$-partition theory \cite{Sta1970,Sta1971,Sta1972}.

Recall that an \emph{integer partition} $\lambda$ is a weakly decreasing sequence of positive integers known as \emph{parts}; the size of this partition, denoted by $|\lambda|$, is the sum of all its parts. We may visually represent a partition $\lambda = (\lambda_1,\lambda_2,\ldots,\lambda_l)$ by its \emph{Young diagram}, which is depicted by boxes placed in rows such that there are $\lambda_i$ boxes in the $i$-th row.

Let $\lambda$ be a fixed partition. A \emph{standard Young tableau of shape $\lambda$} is a filling $T$ of the 
boxes in the Young diagram of $\lambda$ with $1,2,\ldots,|\lambda|$, each used exactly once, such that the 
entries are strictly increasing along rows and down columns. Let  $\SYT(\lambda)$ denote the set of such tableaux.

For $T\in\SYT(\lambda)$ and $1\le i\le |\lambda|$, let
\begin{align*}
	\rho_i=(r_i,s_i):=T^{-1}(i)
\end{align*}
denote the box occupied by $i$ in $T$. Define the \emph{descent set}
\begin{align*}
	\Des(T) := \{i:1\le i\le |\lambda|-1,\ r_{i+1}>r_i\}.
\end{align*}
We require the \emph{descent} statistic
\begin{align*}
	\des(T) := |\Des(T)|,
\end{align*}
and the \emph{comajor index} statistic
\begin{align*}
	\comaj(T) := \sum_{i\in \Des(T)} (|\lambda| - i).
\end{align*}

Our main result is to express Cigler's series $C_k^{(d)}(x,q)$ as the generating function for certain standard Young tableaux of rectangular shapes, thereby indicating Cigler's conjecture. To start with, we let
\begin{align*}
	(k^d) := (\underbrace{k,k,\ldots,k}_{\text{$d$ parts}}),
\end{align*}
the partition into $d$ parts of size $k$. Its Young diagram has a rectangular shape with $d$ rows and $k$ columns.

\begin{theorem}\label{th:main}
	For every positive integers $d$ and $k$,
	\begin{align}\label{eq:gf-main}
		(x;q)_{kd+1} \sum_{n\ge 0} \qangle{n+k}{k}_q^{(d)} x^n = \sum_{T\in \SYT(k^d)} q^{\comaj(T) - k \binom{d}{2}} x^{\des(T) - (d-1)}.
	\end{align}
	Moreover, Cigler's conjecture is true. In particular, the coefficients at the lowest and highest powers of $x$ are respectively
	\begin{align*}
		c_{k,0}^{(d)}(q) &= 1, \qquad\qquad c_{k,(k-1)(d-1)}^{(d)}(q) = q^{2\binom{k}{2}\binom{d}{2}}.
	\end{align*}
\end{theorem}

This note is organized as follows. In Section~\ref{sec:comb}, we interpret Cigler's series $C_k^{(d)}(x,q)$ combinatorially by proving \eqref{eq:gf-main}; we also determine the coefficients of the two endpoint powers of $x$. Next, in Section~\ref{sec:palindromicity}, we establish the palindromicity for all coefficients $c_{k,j}^{(d)}(q)$; our argument is built upon a natural involution on rectangular standard Young tableaux. Finally, we close this paper with further discussions on Fibonacci--Hoggatt numbers and resolve another conjecture of Cigler in Section~\ref{sec:rmk}.

\section{Main Theorem, Part I: Combinatorics}\label{sec:comb}

In this part, we prove \eqref{eq:gf-main}. Our starting point is the following observation.

\begin{lemma}
	For every nonnegative integer $n$ and every positive integers $d$ and $k$,
	\begin{align}\label{eq:gf-middle}
		\qangle{n+k}{k}_q^{(d)} = \sum_{T\in \SYT(k^d)} q^{\comaj(T)-k \binom{d}{2}} \qbinom{(n+d-1)-\des(T)+kd}{kd}_q.
	\end{align}
\end{lemma}

\begin{proof}
	We begin with a generic partition $\lambda$ and a generic nonnegative integer $m$. Note that
	\begin{align*}
		&\sum_{T\in \SYT(\lambda)} q^{\comaj(T)} \qbinom{m-\des(T)+|\lambda|}{|\lambda|}_q\\
		&\qquad = \sum_{T\in \SYT(\lambda)} q^{\des(T)\cdot |\lambda| - \maj(T)} q^{(m-\des(T))\cdot |\lambda|} \qbinom{m-\des(T)+|\lambda|}{|\lambda|}_{q^{-1}},
	\end{align*}
	where the \emph{major index} of $T$ is given by
	\begin{align*}
		\maj(T) := \sum_{i\in \Des(T)} i.
	\end{align*}
	Thus,
	\begin{align*}
		&\sum_{T\in \SYT(\lambda)} q^{\comaj(T)} \qbinom{m-\des(T)+|\lambda|}{|\lambda|}_q\\
		&\qquad = q^{m|\lambda|} \sum_{T\in \SYT(\lambda)} q^{-\maj(T)} \qbinom{m-\des(T)+|\lambda|}{|\lambda|}_{q^{-1}}.
	\end{align*}
	In light of \cite[p.~364, Proposition~7.19.12]{Sta2024}, the sum on the right-hand side of the above can be expressed in terms of a \emph{Schur function}:
	\begin{align}
		\sum_{T\in \SYT(\lambda)} q^{\comaj(T)} \qbinom{m-\des(T)+|\lambda|}{|\lambda|}_q = q^{m|\lambda|} s_\lambda (1,q^{-1},\ldots,q^{-m}).
	\end{align}
	
	According to the above discussions, we have
	\begin{align*}
		\RHS\eqref{eq:gf-middle} = q^{kd(n+d-1) - k\binom{d}{2}} s_{(k^d)}(1,q^{-1},\ldots,q^{-(n+d-1)}).
	\end{align*}
	Now we recall the principal evaluation of Schur functions \cite[p.~374, Theorem~7.21.2]{Sta2024} so that
	\begin{align*}
		s_{(k^d)}(1,q^{-1},\ldots,q^{-(n+d-1)}) &= q^{-k\binom{d}{2}} \prod_{i=1}^d \prod_{j=1}^k \frac{1-q^{-(n+d+j-i)}}{1-q^{-(k+d-i-j+1)}}\\
		&= q^{-k\binom{d}{2}} \prod_{i=1}^d \prod_{j=1}^k \frac{1-q^{-(n+i+j-1)}}{1-q^{-(i+j-1)}}.
	\end{align*}
	Therefore,
	\begin{align*}
		\RHS\eqref{eq:gf-middle} = q^{kd(n+d-1) - 2k\binom{d}{2}} \prod_{i=1}^d \prod_{j=1}^k \frac{1-q^{-(n+i+j-1)}}{1-q^{-(i+j-1)}} = \prod_{i=1}^d \prod_{j=1}^k \frac{1-q^{n+i+j-1}}{1-q^{i+j-1}}.
	\end{align*}
	
	Finally, we know from the definition \eqref{eq:q-Hoggatt-def} that
	\begin{align*}
		\LHS\eqref{eq:gf-middle} = \qangle{n+k}{k}_q^{(d)} = \prod_{j=0}^{d-1} \frac{\qbinom{n+k+j}{k}_q}{\qbinom{k+j}{k}_q} = \prod_{j=1}^d \prod_{i=1}^k \frac{1-q^{n+i+j-1}}{1-q^{i+j-1}},
	\end{align*}
	thereby giving the claimed relation \eqref{eq:gf-middle} by interchanging the indices $i$ and $j$.
\end{proof}

In view of \eqref{eq:gf-middle}, we have
\begin{align*}
	\sum_{n\ge 0} \qangle{n+k}{k}_q^{(d)} x^n = \sum_{T\in \SYT(k^d)} q^{\comaj(T)-k \binom{d}{2}} \sum_{n\ge 0} \qbinom{(n+d-1)-\des(T)+kd}{kd}_q x^n.
\end{align*}
To simplify the inner sum on the right-hand side, we need to determine the range of $\des(T)$.

\begin{lemma}\label{le:des-bound}
	For standard Young tableaux $T\in \SYT(k^d)$, we have
	\begin{align}\label{eq:des-bound}
		d-1 \le \des(T) \le k(d-1).
	\end{align}
	Moreover, both inequalities are attainable, each by a unique standard Young tableau.
\end{lemma}

\begin{proof}
	For the lower bound, we assume that in the standard Young tableau $T$, the entries in the first column are $u_1 < u_2 < \cdots < u_d$. For each $u_{i+1}-1$ with $1\le i\le d-1$, it is clear that its row index must be smaller than that of $u_{i+1}$, thereby implying that $u_{i+1}-1$ is a descent. Hence,
	\begin{align*}
		\des(T) \ge \big|\{u_2-1,u_3-1,\ldots, u_d-1\}\big| = d-1.
	\end{align*}
	On the other hand, suppose the entries in the last column are $u'_1 < u'_2 < \cdots < u'_d$. It is also true that all $u'_1,\ldots,u'_{d-1}$ are descents. To ensure $\des(T) = d-1$, we must have $u'_i = u_{i+1} - 1$ for $1\le i\le d-1$. In this situation, the only possibility is the standard Young tableau $T^\dagger$ of shape $(k^d)$ in which the entry in the $i$-th row and $j$-th column is $j+(i-1)k$.
	
	For the upper bound, we assume that the first row of the standard Young tableau $T$ is filled with $v_1 < v_2 < \cdots < v_k$. For each $v_{j+1}-1$ with $1\le j\le k-1$, its row index is at least $1$, which is never smaller than that of $v_{j+1}$. Thus,
	\begin{align*}
		\des(T) \le \big|\{1,2,\ldots,kd-1\}\backslash \{v_2-1,v_3-1,\ldots, v_k-1\}\big| = k(d-1).
	\end{align*}
	Suppose further that the entries in the last row are $v'_1 < v'_2 < \cdots < v'_k$. It is clear that $v'_1,\ldots,v'_{k-1}$ are non-descents. Thus, to ensure $\des(T) = k(d-1)$, we must have $v'_j = v_{j+1}-1$. Then the only possible standard Young tableau $T^\ddagger$ is such that the $(i,j)$-th box is filled with $i+(j-1)d$.
\end{proof}

According to \eqref{eq:des-bound}, for every $T\in \SYT(k^d)$,
\begin{align*}
	(d-1) - \des(T) \le 0.
\end{align*}
Thus,
\begin{align*}
	\sum_{n\ge 0} \qbinom{(n+d-1)-\des(T)+kd}{kd}_q x^n = x^{\des(T)-(d-1)} \sum_{i\ge 0} \qbinom{i+kd}{kd}_q x^{i}.
\end{align*}
Now we can apply the following evaluation of \emph{$q$-binomial series} \cite[p.~36, eq.~(3.3.7)]{And1998}:
\begin{align*}
	\sum_{i\ge 0} \qbinom{i+j}{j}_q x^i = \frac{1}{(x;q)_{j+1}},
\end{align*}
from which we get
\begin{align*}
	\sum_{n\ge 0} \qbinom{(n+d-1)-\des(T)+kd}{kd}_q x^n = \frac{x^{\des(T)-(d-1)}}{(x;q)_{kd+1}}.
\end{align*}

Finally, we conclude that
\begin{align*}
	\sum_{n\ge 0} \qangle{n+k}{k}_q^{(d)} x^n = \frac{1}{(x;q)_{kd+1}} \sum_{T\in \SYT(k^d)} q^{\comaj(T)-k \binom{d}{2}} x^{\des(T)-(d-1)},
\end{align*}
thereby producing the desired relation \eqref{eq:gf-main}. In addition, we have the coefficient of the lowest power of $x$:
\begin{align*}
	c_{k,0}^{(d)}(q) = q^{\comaj(T^\dagger)-k \binom{d}{2}} = 1,
\end{align*}
while the coefficient of the highest power of $x$ is
\begin{align*}
	c_{k,(k-1)(d-1)}^{(d)}(q) = q^{\comaj(T^\ddagger)-k \binom{d}{2}} = q^{2\binom{k}{2}\binom{d}{2}};
\end{align*}
here the standard Young tableaux $T^\dagger$ and $T^\ddagger$ are as in the proof of Lemma~\ref{le:des-bound}.

\section{Main Theorem, Part II: Palindromicity}\label{sec:palindromicity}

A direct indication of the relation \eqref{eq:gf-main} and the bounds \eqref{eq:des-bound} is that Cigler's series $C_k^{(d)}(x,q)$ is in $\mathbb{N}[q,q^{-1}][x]$, while the highest exponent of $x$ is $(k-1)(d-1)$. Now we shall show that the coefficients $c_{k,j}^{(d)}(q)$ of $x^j$ in $C_k^{(d)}(x,q)$ are indeed \emph{polynomials} in $\mathbb{N}[q]$ instead of Laurent polynomials. For this purpose, we have to bound $\comaj(T)$ for standard Young tableaux $T\in \SYT(k^d)$.

\begin{lemma}
	For standard Young tableaux $T\in \SYT(k^d)$, we have
	\begin{align}\label{eq:comaj-bound}
		\comaj(T) \ge k \binom{d}{2}.
	\end{align}
\end{lemma}

\begin{proof}
	We rewrite $\comaj(T)$ as
	\begin{align*}
		\comaj(T) = \sum_{i\in \Des(T)} (kd - i) = \sum_{m=1}^{kd} \big|\{l\in \Des(T): l<m\}\big|.
	\end{align*}
	Now we prove that, for each $m$, there are at least $r_m-1$ descents smaller than $m$, where $m$ is placed in the $(r_m,s_m)$-th box. In the $s_m$-th cloumn, we assume that the first $r_m$ entries are $u_1 < u_2 < \cdots < u_{r_m} =: m$. It is clear that in each interval $\{u_i,u_i+1,\ldots, u_{i+1}-1\}$ with $1\le i\le r_m-1$, there exists at least one descent $l_i$. Moreover, $l_i \le u_{i+1}-1 \le u_{r_m} - 1 < m$. Since there are $r_m-1$ such nonintersecting intervals, we have
	\begin{align*}
		\big|\{l\in \Des(T): l<m\}\big| \ge r_m-1.
	\end{align*}
	Finally,
	\begin{align*}
		\comaj(T) \ge \sum_{m=1}^{kd} (r_m - 1) = k \sum_{i=1}^d (i-1) = k \binom{d}{2},
	\end{align*}
	as claimed.
\end{proof}

Now we have seen that
\begin{align*}
	C_k^{(d)}(x,q) \in \mathbb{N}[q][x].
\end{align*}
Moreover, the coefficients $c_{k,j}^{(d)}(q)$ are given by
\begin{align*}
	c_{k,j}^{(d)}(q) = \sum_{\substack{T \in \SYT(k^d)\\\des(T)=j+d-1}} q^{\comaj(T)-k \binom{d}{2}}.
\end{align*}
To show these coefficients are palindromic, we need the following observation.

\begin{lemma}
	Define a map $\theta: \SYT(k^d) \to \SYT(k^d)$ by
	\begin{align*}
		\theta(T)(r,s) := kd + 1 - T(d+1-r, k+1-s).
	\end{align*}
	Then $\theta$ is an involution. In addition,
	\begin{align}\label{eq:theta-des}
		\des(\theta(T)) = \des(T),
	\end{align}
	and
	\begin{align}\label{eq:theta-comaj}
		\comaj(\theta(T)) = kd \des(T) - \comaj(T).
	\end{align}
\end{lemma}

\begin{proof}
	It is a direct verification that $\theta$ is an involution on $\SYT(k^d)$. Now we check the two relations on statistics. First,
	\begin{align*}
		\Des(\theta(T)) = \{kd-i: i\in \Des(T)\},
	\end{align*}
	so that \eqref{eq:theta-des} holds. Next,
	\begin{align*}
		\comaj(\theta(T)) = \sum_{i'\in \Des(\theta(T))} (kd-i') = \sum_{i\in \Des(T)} i,
	\end{align*}
	thereby leading us to \eqref{eq:theta-comaj}.
\end{proof}

Note that the map $\theta$, if restricted to the subset of standard Young tableaux
\begin{align*}
	\{T\in \SYT(k^d): \des(T)=j+d-1\},
\end{align*}
is still an involution. Moreover,
\begin{align*}
	\left(\comaj(T)-k \binom{d}{2}\right) + \left(\comaj(\theta(T))-k \binom{d}{2}\right) = kd(j+d-1) - 2k \binom{d}{2} = kdj,
\end{align*}
which is independent of $T$. These facts tell us that in $c_{k,j}^{(d)}(q)$, the sum of the lowest and highest exponents of $q$ is $kdj$. Also, by pairing $T$ and $\theta(T)$, we have
\begin{align*}
	c_{k,j}^{(d)}(q) = q^{kdj} c_{k,j}^{(d)}(q^{-1}),
\end{align*}
and hence conclude the palindromicity.

\section{Closing remarks}\label{sec:rmk}

In \cite{Cig2021}, Cigler also considered a family of Fibonacci--Hoggatt numbers. For nonzero numbers $a$ and $b$, let $(F_{a,b}(n))_{n\ge 0}$ be a sequence defined by the recurrence
\begin{align*}
	F_{a,b}(n) = aF_{a,b}(n-1) + bF_{a,b}(n-2),
\end{align*}
together with the initial values $F_{a,b}(0) = 0$ and $F_{a,b}(1) = 1$. When $a=b=1$, the sequence reduces to the conventional \emph{Fibonacci numbers}.

In analogy to \eqref{eq:q-Hoggatt-def}, define the \emph{Fibonacci--Hoggatt numbers} by
\begin{align*}
	\qangle{n}{k}_{F_{a,b}}^{(d)} := \frac{\langle n\rangle_{F_{a,b}}^{(d)}!}{\langle k\rangle_{F_{a,b}}^{(d)}! \langle n-k\rangle_{F_{a,b}}^{(d)}!} = \prod_{j=0}^{d-1} \frac{\qbinom{n+j}{k}_{F_{a,b}}}{\qbinom{k+j}{k}_{F_{a,b}}},
\end{align*}
where
\begin{align*}
	\langle n\rangle_{F_{a,b}}^{(d)} := \qbinom{n+d-1}{d}_{F_{a,b}},\qquad\qquad \langle n\rangle_{F_{a,b}}^{(d)}! := \prod_{j=1}^n \langle j\rangle_{F_{a,b}}^{(d)},
\end{align*}
with
\begin{align*}
	\qbinom{n}{k}_{F_{a,b}}:=\begin{cases}
		\displaystyle \prod_{j=0}^{k-1} \frac{F_{a,b}(n-j)}{F_{a,b}(k-j)}, & \text{if $0\le k\le n$},\\[10pt]
		0, & \text{otherwise}.
	\end{cases}
\end{align*}

Note that the numbers $F_{a,b}(n)$ can be expressed by \emph{Binet's formula}:
\begin{align*}
	F_{a,b}(n) = \frac{\alpha^n - \beta^n}{\alpha-\beta},
\end{align*}
where
\begin{align*}
	\alpha := \frac{a+\sqrt{a^2+4b}}{2}, \qquad\qquad \beta := \frac{a-\sqrt{a^2+4b}}{2}.
\end{align*}
Putting
\begin{align*}
	q:= \frac{\beta}{\alpha},
\end{align*}
and noting that $\alpha\beta = -b$, we have
\begin{align*}
	F_{a,b}(n) = \left(-\frac{b}{q}\right)^{\frac{n-1}{2}} \frac{1-q^n}{1-q},
\end{align*}
so that
\begin{align*}
	\qbinom{n}{k}_{F_{a,b}} = \left(-\frac{q}{b}\right)^{\frac{k(k-n)}{2}} \qbinom{n}{k}_q.
\end{align*}

In \cite{Cig2021}, Cigler considered the following series
\begin{align}\label{eq:G-def}
	G_{k,F_{a,b}}^{(d)}(x) := \left(\sum_{m=0}^{kd+1}(-1)^{\binom{m+1}{2}} b^{\binom{m}{2}} \qbinom{kd+1}{m}_{F_{a,b}} x^m\right) \left(\sum_{n\ge 0} \qangle{n+k}{k}_{F_{a,b}}^{(d)} x^n\right).
\end{align}
If we put
\begin{align*}
	y:= \left(-\frac{b}{q}\right)^{\frac{kd}{2}} x,
\end{align*}
it follows from a direct computation that
\begin{align*}
	G_{k,F_{a,b}}^{(d)}(x) = (y;q)_{kd+1} \sum_{n\ge 0} \qangle{n+k}{k}_q^{(d)} y^n,
\end{align*}
where we have applied the following evaluation \cite[p.~36, eq.~(3.3.6)]{And1998} to the first sum in \eqref{eq:G-def}:
\begin{align*}
	\sum_{i=0}^j \qbinom{j}{i}_q (-1)^i z^i q^{\binom{i}{2}} = (z;q)_j.
\end{align*}

Note that $y$ and $x$ only differ by a scalar. By virtue of Theorem~\ref{th:main}, we may immediately confirm the following conjecture of Cigler~\cite[Conjecture~18]{Cig2021}, which also reduces to \cite[Conjecture~14]{Cig2021} at the $a=b=1$ specialization.

\begin{theorem}
	$G_{k,F_{a,b}}^{(d)}(x)$ is a polynomial in $x$ of degree $(k-1)(d-1)$ with constant term being $1$ and the coefficient of the highest power of $x$ being $b^{2\binom{k}{2}\binom{d}{2}}$.
\end{theorem}

\subsection*{Acknowledgements}

Shane Chern was supported by the Austrian Science Fund (No.~10.55776/F1002). Wenle Shi was supported by the China Scholarship Council (No.~202506060077).

\bibliographystyle{amsplain}

\end{document}